\documentclass{amsart}
\pdfoutput=1

\usepackage{cmap}
\usepackage{amsmath}
\usepackage[a4paper]{geometry}
\usepackage{amssymb}
\usepackage{amsthm}
\usepackage{enumerate}
\usepackage{verbatim}
\usepackage{perpage}
\usepackage{datetime}
\usepackage{nnfootnote}
\usepackage{color}
\usepackage{mathabx}
\usepackage{hyphenat}
\usepackage[pdftex, bookmarks, pdfpagelabels]{hyperref}
\newcommand{\papertitle}{Pointwise Convergence of Dyadic Partial Sums of Almost Periodic Fourier Series}
\definecolor{linkcolor}{rgb}{0.2,0,0.5}
\hypersetup{colorlinks=true,
						linkcolor=linkcolor,
						citecolor=linkcolor, 
						pdftitle={\papertitle}, 
						pdfauthor={Andrew David Bailey}, 
						pdfsubject={Fourier Analysis}, 
						pdfkeywords={Almost periodic functions, Stepanov space, Maximal operator, Littlewood--Paley theory, Hilbert transform, pointwise convergence}, 
						pdfstartview=FitH,
						pdfdisplaydoctitle=true,
						pdfpagemode=UseNone,
						pdfborder=0 0 0}
\pdfpageattr {/Group << /S /Transparency /I true /CS /DeviceRGB>>}
\usepackage[T2A, OT1]{fontenc}
\usepackage{fancyhdr}

\allowdisplaybreaks

\theoremstyle{plain}
 \newtheorem{thm}{Theorem}[section] 
 \newtheorem*{thm*}{Theorem}
 
 \newtheorem{lemma}[thm]{Lemma}
 \newtheorem{prop}[thm]{Proposition}
 \newtheorem*{prop*}{Proposition}

 \newtheorem{defn}[thm]{\bf Definition} 
 \newtheorem*{defn*}{\bf Definition}

\newdateformat{adbdate}{%
\dayofweekname{\THEDAY}{\THEMONTH}{\THEYEAR}  {\ordinal{DAY}} \monthname[\THEMONTH] \THEYEAR}
  \adbdate
  \settimeformat{xxivtime}
  
\renewcommand{\|}{\displaystyle}
\renewcommand{\(}{\left(}
\renewcommand{\)}{\right)}
\newcommand{\fhat}{\widehat{f}}

\renewcommand{\Gamma}{\varGamma}
\renewcommand{\epsilon}{\varepsilon}

\renewcommand{\leq}{\leqslant}
\renewcommand{\geq}{\geqslant}
\MakePerPage{footnote} 

\DeclareMathOperator{\supp}{supp}
\DeclareMathOperator{\sgn}{sgn}

\newcommand{\supcite}[2][]{%
  \ifthenelse{\isempty{#1}}%
    {$^{\textrm{\cite{#2}}}$}%
    {$^{\textrm{\cite[#1]{#2}}}$}%
}
\newcommand{\notlabel}[1]{} 

\DeclareMathOperator{\loc}{loc}

\begin{document}
\title[\papertitle]{\papertitle}
\author[]{Andrew D. Bailey}
\address{Andrew D. Bailey, School of Mathematics, Watson Building, University of Birmingham, Edgbaston, Birmingham, B15 2TT, United Kingdom.}
\email{baileya@maths.bham.ac.uk}
\nnfoottext{Last updated on {\today} at \currenttime.
The author was supported by an EPSRC doctoral training grant.  This work has appeared previously as part of the author's MPhil thesis, \cite{mphil}.  The author would like to express his gratitude to his MPhil and PhD supervisor, Jonathan Bennett, for all his support and assistance.}
\begin{abstract}
It is a classical result that dyadic partial sums of the Fourier series of functions $\| f \in L^p(\mathbb{T})$ converge almost everywhere for $p \in (1, \infty)$.  In 1968, E. A. Bredihina established an analogous result for functions belonging to the Stepanov space of almost periodic functions $S^2$ whose Fourier exponents satisfy a natural separation condition.  Here, the maximal operator corresponding to dyadic partial summation of almost periodic Fourier series is bounded on the Stepanov spaces $S^{2^k}$, $k \in \mathbb{N}$ for functions satisfying the same condition; Bredihina's result follows as a consequence.  In the process of establishing these bounds, some general results are obtained which will facilitate further work on operator bounds and convergence issues in Stepanov spaces.  These include a boundedness theorem for the Hilbert transform and a theorem of Littlewood--Paley type.  An improvement of ``$S^{2^k}$, $k \in \mathbb{N}$'' to ``$S^p$, $p \in (1, \infty)$'' is also seen to follow from a natural conjecture on the boundedness of the Hilbert transform.
\end{abstract}
\maketitle

\section{Introduction}
In 1924, forty two years before the appearance of Carleson's landmark paper, \cite{carleson}, establishing the pointwise convergence of Fourier series for functions in $L^2(\mathbb{T})$, Kolmogorov proved in \cite{kol} that for almost every $x \in \mathbb{T}$,
\begin{equation*}
f(x) = \lim_{k \rightarrow \infty} \sum_{|n| \leq 2^k} \fhat(n) e^{inx}
\end{equation*}
whenever $f$ is a function in $L^2(\mathbb{T})$, that is to say that dyadic partial sums of Fourier series for $L^2$ functions converge almost everywhere.  This result can be generalised to $f \in L^p(\mathbb{T})$ for $p \in (1, \infty)$ using Littlewood--Paley theory (see, for example, pp.\ 374-375 in \cite{grafakosclassical}).  In 1968, E. A. Bredihina proved in \cite{bred} the following generalisation of Kolmogorov's work to the context of almost periodic Fourier series:
\begin{thm}[E. A. Bredihina, 1968] \label{bredthm}
Let $f$ be an almost periodic function in the Stepanov space $S^2$ with Fourier exponents separated by some fixed constant $\alpha_f > 0$.  Then for almost every $x \in \mathbb{R}$,
\begin{equation*}
f(x) = \lim_{k \rightarrow \infty} \sum_{|\lambda_n| \leq 2^k} \fhat(\lambda_n) e^{i\lambda_n x}.
\end{equation*}
\end{thm}
Here, for $p \in [1, \infty)$, the Stepanov spaces $S^p$ can be defined as the completion of the space of trigonometric polynomials on $\mathbb{R}$,
\begin{equation*}
\mathcal{P} \coloneq \Big\{ f = \sum_{n = -N}^N c_n e^{i \lambda_n \cdot} : (c_n) \subseteq \mathbb{C}, (\lambda_n) \subseteq \mathbb{R}, N \in \mathbb{N} \Big\}
\end{equation*}
with respect to the Stepanov norm, defined as
\begin{equation*}
\Vert f \Vert_{S^p} \coloneq \Big( \sup_{x \in \mathbb{R}} \int_x^{x+1} |f(s)|^p \, ds \Big)^{\frac{1}{p}}.
\end{equation*}
This space contains all $L^p$ periodic functions of any period.

A function $f \in S^p$ has a unique associated Fourier series,
\begin{equation*}
\sum_{n \in \mathbb{Z}} \fhat(\lambda_n) e^{i\lambda_n \cdot},
\end{equation*}
where the Fourier coefficients are defined for each $\lambda \in \mathbb{R}$ as
\begin{equation*}
\fhat(\lambda) \coloneq \lim_{T \rightarrow \infty} \frac{1}{2T} \int_{-T}^T f(x) e^{-i\lambda x} \, dx.
\end{equation*}
For any given almost periodic function $f$, this quantity is non-zero for only countably many choices of $\lambda$.  As a matter of convention, the sequence of Fourier exponents, $(\lambda_n)_{n \in \mathbb{Z}} \subseteq \mathbb{R}$, will be chosen here to be the unique strictly increasing sequence such that for each $n \in \mathbb{Z}$, $\lambda_{-n} = -\lambda_n$ and at least one of $\fhat(\lambda_n)$ and $\fhat(\lambda_{-n})$ is non-zero.  The function $f$ will be said to satisfy the separation condition if there exists $\alpha > 0$ such that $\lambda_{n+1} - \lambda_n > \alpha$ for all $n \in \mathbb{Z}$.  The infimum of all possible choices of such $\alpha$ will be denoted by $\alpha_f$.

The Stepanov spaces satisfy a nesting property in the sense that for $p_1$, $p_2 \in [1, \infty)$ with $p_1 \leq p_2$, it is the case that $S^{p_2} \subseteq S^{p_1}$ and that for any $f \in S^{p_1}$, $\| \Vert f \Vert_{S^{p_1}} \leq \Vert f \Vert_{S^{p_2}}$.  However, the reader is warned that for all $p \in [1, \infty)$, there exist functions $f \in L^{p}_{\loc}(\mathbb{R})$ with $\Vert f \Vert_{S^p} < \infty$ such that $f \notin S^p$.

For further information on almost periodic functions, the reader is referred to \cite{mphil}, \cite{levitan}, \cite{hier}, \cite{besi} and \cite{bohr}.

The main result of this paper is the following:
\begin{thm} \label{mainthm}
For each $j \in \mathbb{N}$, define the dyadic Fourier partial sum operator acting on almost periodic trigonometric polynomials $f$ as
\begin{equation*}
S_j f \coloneq \sum_{|\lambda_n| \leq 2^j} \fhat(\lambda_n) e^{i\lambda_n \cdot}
\end{equation*}
with corresponding maximal operator $\| S^{*}f \coloneq \sup_{j \in \mathbb{N}} |S_jf|$.  Then for each $k \in \mathbb{N}$, the operator $\| S^{*}$ extends to the class of all functions $f \in S^{2^k}$ that satisfy the separation condition and satisfies the bound
\begin{equation*}
\Vert S^{*}f \Vert_{S^{2^k}} \lesssim_{\alpha_f} \Vert f \Vert_{S^{2^k}}.
\end{equation*}
\end{thm}
Here and throughout this paper, the symbol $\lesssim$ is used to signify that the left hand side is bounded above by a constant multiple of the right hand side, with this constant independent of $f$.  Subscripts indicate explicitly dependence of the constant on parameters.

It is seen in Section \ref{secauxresults} that Theorem \ref{mainthm} implies almost everywhere convergence of dyadic partial sums of Fourier series for functions in $S^{2^k}$ that satisfy the separation condition for each $k \in \mathbb{N}$ which, by the nesting of Stepanov spaces described above, is logically equivalent to Theorem \ref{bredthm} of Bredihina.  Section \ref{secauxresults} also contains other technical results required for the remainder of this paper.  In Section \ref{secht}, the Hilbert transform for functions in the Stepanov spaces is considered and a boundedness theorem is proved, whilst in Section \ref{seclp}, a theorem of Littlewood--Paley type is established for the Stepanov spaces.  Both of these theorems are natural analogues of fundamental theorems from standard $L^p(\mathbb{R})$ theory and will be of interest beyond the context of Theorem \ref{mainthm}.  With these results established, the proof of Theorem \ref{mainthm} is provided in Section \ref{secmainproof}.  To extend Theorem \ref{mainthm} to $S^p$ for all $p \in (1, \infty)$, it remains only to improve the result on the boundedness of the Hilbert transform from Section \ref{secht}; this is discussed in Section \ref{secfurtherremarks}.
\section{Auxiliary Results} \label{secauxresults}
For $p$, $q \in [1, \infty)$, there is a natural notion of weak $p$--$q$ boundedness of an operator $T$ with respect to Stepanov norms, namely that
\begin{equation*}
\sup_{x \in \mathbb{R}}  | \{s \in [x, x+1] : | Tf(s) | > \lambda \} |  \lesssim \(\frac{\Vert f \Vert_{S^p}}{\lambda}\)^q 
\end{equation*}
for all $f \in S^p$ and $\| \lambda > 0$.  By Chebyshev's inequality, it is easy to see that this is a genuinely weaker bound than $\| \Vert Tf \Vert_{S^q} \lesssim \Vert f \Vert_{S^p}$.  Weak boundedness of maximal operators in Stepanov spaces can be used to show pointwise convergence results, as is the case in the familiar $L^p$ setting.  Indeed, the following result holds:
\begin{thm} \label{maxboundconv}
Let $\| (T_j)_{j \in \mathbb{N}}$ be a family of linear operators on $\| S^p$, $\| p \in [1, \infty)$ and assume that $\| T^*f \coloneq \sup_{j \in \mathbb{N}} |T_jf|$ is weakly bounded $p$--$q$ for some $\| q \in [1, \infty)$.  Then the set
\begin{equation*}
E \coloneq \{ f \in S^p: \lim_{j \rightarrow \infty} T_j f(x) = f(x) \mbox{ \emph{a.e.}}  \}
\end{equation*}
is closed in $\| S^p$.
\end{thm}
This can be proved by a straightforward adaptation of standard arguments such as those in the proof of Theorem 2.2 in \cite{duo}.  This involves choosing a sequence $(f_n)_{n \in \mathbb{N}} \subseteq E$ such that $\Vert f_n - f \Vert_{S^p} \rightarrow 0$ as $n \rightarrow \infty$ for some $f \in S^p$.  Then for each $x \in \mathbb{R}$ and $\lambda > 0$, it can be shown that
\begin{eqnarray*}
&& |\{ s \in [x, x + 1] : \limsup_{j \rightarrow \infty} |T_jf(s) - f(s)| > \lambda \} |\\*
&\leq& |\{s \in [x, x+1] : T^{*}(f-f_n)(s) > \tfrac{\lambda}{2}\}| + \int_x^{x+1} \frac{2^p|(f-f_n)(s)|^p}{\lambda^p} \, ds
\end{eqnarray*}
from which it can be deduced that $f \in E$ as a consequence of weak boundedness of $T^{*}$.  As in the $L^p$ setting, this theorem is useful as there is a natural dense subspace of $S^p$, namely $\mathcal{P}$, in which reasonable pointwise convergence results usually either hold trivially or are easily verified.  In the context of the present paper, for example, it is certainly trivial that the Fourier series associated to trigonometric polynomials converge pointwise in every reasonable sense.  Since the proof of Theorem \ref{maxboundconv} naturally adapts to subspaces of $S^p$ equipped with the subspace topology, by the fact that $S^{2^k} \subseteq S^2$ for all $k \in \mathbb{N}$, Theorem \ref{bredthm} follows as a consequence of Theorem \ref{mainthm}.

There is a natural analogue of Parseval's Identity for almost periodic functions\footnote{This theorem is most natural on the Besicovitch space of almost periodic functions $B^2$ since the expression $\| \lim_{T \rightarrow \infty} \Big( \frac{1}{2T} \int_{-T}^T |f(x)|^2 \, dx \Big)^{\frac{1}{2}}$ corresponds to $\| \Vert f \Vert_{B^2}$.  The Besicovitch spaces strictly contain the Stepanov spaces.}:
\begin{thm}[Parseval's Identity] \label{parseval}
For any $\| f \in S^2$,
\begin{equation*}
\lim_{T \rightarrow \infty} \frac{1}{2T}\int_{-T}^T |f(x)|^2 \, dx = \sum_{n \in \mathbb{N}} |\fhat(\lambda_n)|^2.
\end{equation*}
\end{thm}
For a proof of this result, the reader is referred to \cite{besi}, p. 109.  It is easy to see that for any $f \in S^2$,
\begin{equation*}
\Big(\lim_{T \rightarrow \infty} \frac{1}{2T}\int_{-T}^T |f(x)|^2 \, dx \Big)^{\frac{1}{2}} \leq \Vert f \Vert_{S^2}
\end{equation*}
so a Parseval-type inequality,
\begin{equation*}
\Big(\sum_{n \in \mathbb{N}} |\fhat(\lambda_n)|^2 \Big)^{\frac{1}{2}} \leq \Vert f \Vert_{S^2},
\end{equation*}
immediately follows from the above.

There is a partial converse to this inequality which can be stated as follows:
\begin{prop} \label{bred}
Let $f$ be a trigonometric polynomial.  Then for any $\| x \in \mathbb{R}$,
\begin{equation*}
\int^{x+1}_{x} |f(s)|^2 \, ds \leqslant \Big(\frac{2\pi}{\alpha_f} + 1\Big) \sum_{n \in \mathbb{Z}} | \fhat(\lambda_n) |^2.
\end{equation*}
\end{prop}
The proof of this proposition will require the following lemma:
\begin{lemma} \label{hilmore}
Let $\| (\lambda_k)_{k \in \mathbb{Z}} \subseteq \mathbb{R}$ be an increasing sequence such that there exists $\| \alpha > 0$ so that $\| \lambda_{k+1} - \lambda_k >
\alpha$ for all $\| k \in \mathbb{N}$.  Let $T$ be the operator acting on sequences $\| (a_j)_{j \in \mathbb{Z}} \in \ell^2$ defined such that $\| (T(a_j))_j \coloneq \sum_{k \in \mathbb{Z} \setminus \{j\}} \frac{a_k}{\lambda_j - \lambda_k}$.  Then 
\begin{equation*}
\Vert T(a_j) \Vert_{\ell^2} \leqslant \frac{\pi}{\alpha} \Vert (a_j) \Vert_{\ell^2}.
\end{equation*}
\end{lemma}
A proof of this result, which is a generalisation of a famous inequality of Hilbert and Schur, can be found on pp.\ 138-140 of \cite{montgom}.
\begin{proof}[Proof of Proposition \ref{bred}]
Fix $x \in \mathbb{R}$ and consider that
\begin{eqnarray*}
&& \int^{x+1}_x \Big| \sum_{n \in \mathbb{Z}} \fhat(\lambda_n) e^{i\lambda_ns} \Big|^2 \, ds \\*
&=& \sum_{n \in \mathbb{Z}} \sum_{m \in \mathbb{Z}} \int^{x+1}_x \fhat(\lambda_n) \overline{\fhat(\lambda_m)} e^{i(\lambda_n-\lambda_m)s} \, ds \\
&=& \sum_{n \in \mathbb{Z}} \underset{m \neq n}{\sum_{m \in \mathbb{Z}}} \fhat(\lambda_n) \overline{\fhat(\lambda_m)} \frac{1}{i(\lambda_n-\lambda_m)} (e^{i(\lambda_n-\lambda_m)(x+1)} - e^{i(\lambda_n-\lambda_m)x}) + \sum_{n \in \mathbb{Z}} |\fhat(\lambda_n)|^2\\
&\leq& \Big| \sum_{n \in \mathbb{Z}} \underset{m \neq n}{\sum_{m \in \mathbb{Z}}} \fhat(\lambda_n) \overline{\fhat(\lambda_m)} \frac{1}{i(\lambda_n-\lambda_m)} e^{i(\lambda_n-\lambda_m)(x+1)} \Big|\\*
&& \quad {} + \Big|\sum_{n \in \mathbb{Z}} \underset{m \neq n}{\sum_{m \in \mathbb{Z}}} \fhat(\lambda_n) \overline{\fhat(\lambda_m)} \frac{1}{i(\lambda_n-\lambda_m)}e^{i(\lambda_n-\lambda_m)x} \Big| + \sum_{n \in \mathbb{Z}} |\fhat(\lambda_n)|^2.
\end{eqnarray*}
It is noted that the interchange of sums and integrals above is permitted as all sums possess only finitely many non-zero terms since $f \in \mathcal{P}$.  Writing $a_n \coloneq \fhat(\lambda_n) e^{i\lambda_n (x+1)}$ and letting $T$ denote the operator defined in Lemma \ref{hilmore}, by the Cauchy--Schwarz inequality and Lemma \ref{hilmore}, the first term here is bounded above by
\begin{eqnarray*}
&& \sum_{n \in \mathbb{Z}} | a_n \overline{(T(a_m))_n} |\\*
&\leq& \frac{\pi}{\alpha_f} \Vert (a_n) \Vert_{\ell^2}^2\\
&=& \frac{\pi}{\alpha_f} \Vert (\fhat(\lambda_n)) \Vert_{\ell^2}^2.
\end{eqnarray*}
Since the second term can be treated identically, the result follows.
\end{proof}
This section will be concluded with the following result:
\begin{thm} \label{bredext}
Let $\| \psi \in L^1(\mathbb{R})$ have real-valued, continuous and bounded Fourier transform,
\begin{equation*}
\widehat{\psi}(\xi) = \int_{\mathbb{R}} \psi(x) e^{-ix\xi} \, dx.
\end{equation*}
Then for $\| f \in \mathcal{P}$,
\begin{equation*}
 f * \psi = \sum_{n \in \mathbb{Z}} \fhat(\lambda_n) \widehat{\psi}(\lambda_n) e^{i \lambda_n \cdot},
\end{equation*}
where $*$ represents convolution on the line, that is $\| f * \psi = \int_{\mathbb{R}} f(\cdot -y) \psi(y) \, dy$.
\end{thm}
\begin{proof}
This is a consequence of the following straightforward calculation:
\begin{eqnarray*}
f * \psi &=& \int_{\mathbb{R}} \Big( \sum_{n \in \mathbb{Z}} \fhat(\lambda_n) e^{i \lambda_n ( \cdot - y)} \Big) \psi(y) \, dy\\
&=& \sum_{n \in \mathbb{Z}} \fhat(\lambda_n) e^{i \lambda_n \cdot} \int_{\mathbb{R}} \psi(y) e^{-i \lambda_n y} \, dy\\
&=& \sum_{n \in \mathbb{Z}} \fhat(\lambda_n) \widehat{\psi}(\lambda_n) e^{i \lambda_n \cdot}.
\end{eqnarray*}
As before, the interchange of the sum and integral is permitted here as the sum possesses only finitely many non-zero terms since $\| f \in \mathcal{P}$.
\end{proof}

\section{The Hilbert Transform on Almost Periodic Functions} \label{secht}
The definition of the Hilbert transform naturally extends to the context of almost periodic functions:
\begin{defn} \label{aphilberttransform}
For $\| f \in S^p$, $\| p \in [1, \infty)$, the Hilbert transform is defined as
\begin{equation*}
Hf(x) \coloneq \mbox{p.v. }\frac{1}{\pi}\int_{\mathbb{R}} \frac{f(y)}{x-y} \, dy. 
\end{equation*}
\end{defn}

Integration in the complex plane over a suitable indented contour shows that the Hilbert transform of an almost periodic function is an almost periodic function with Fourier coefficients $\| \widehat{Hf}(\lambda) = -i \sgn(\lambda) \fhat(\lambda)$ for each $\| \lambda \in \mathbb{R}$.

It turns out that for the present purposes, it will be useful to consider slightly modified versions of the Hilbert transform:

\begin{defn}
For $\| f \in S^p$ and $\| p \in [1, \infty)$, define $\| H_{\pm}$ to be the operator such that $\| \widehat{H_{\pm}f}(\lambda) = -i \sgn_{\pm}(\lambda) \fhat(\lambda)$ for any $\| \lambda \in \mathbb{R}$, where the functions $\| \sgn_{\pm} : \mathbb{R} \rightarrow \{ -1, 1 \}$ are given by
\begin{eqnarray*}
sgn_{+}(x) &\coloneq& \left\{
\begin{array}{r l}
1, & x \in [0, \infty)\\
-1, & x \in (-\infty, 0),
\end{array} \right.\\
sgn_{-}(x) &\coloneq& \left\{
\begin{array}{r l}
1, & x \in (0, \infty)\\
-1, & x \in (-\infty, 0].
\end{array} \right.\\
\end{eqnarray*}
\end{defn}
By noting that for each $p \in [1, \infty)$,
\begin{equation*}
|\fhat(0)| \leq \lim_{T \rightarrow \infty} \frac{1}{2T} \int_{-T}^T |f(x)| \, dx \leq \Vert f \Vert_{S^p},
\end{equation*}
it is clear that boundedness of $H_{\pm}$ on Stepanov spaces is equivalent to boundedness of $H$ on Stepanov spaces.

The following lemma shows that a well-known identity for the usual Hilbert transform adapts well to the context of the modified Hilbert transforms given above acting on almost periodic trigonometric polynomials:

\begin{lemma} \label{niceformula}
Let $f \in \mathcal{P}$.  Then
\begin{equation*}
(H_{\pm}(f))^2 = f^2 + 2H_{\pm}(fH_{\pm}(f)).
\end{equation*}
\end{lemma}
\begin{proof}
First note that
\begin{eqnarray*}
&& 2H_{\pm}(fH_{\pm}(f))\\*
&=& 2H_{\pm}\Big(\sum_{n \in \mathbb{Z}} \sum_{m \in \mathbb{Z}} (-i \sgn_{\pm}(\lambda_m))\fhat(\lambda_n) \fhat(\lambda_m) e^{i(\lambda_n
+ \lambda_m)\cdot}\Big)\\
&=& 2 \sum_{n \in \mathbb{Z}} \sum_{m \in \mathbb{Z}} -\sgn_{\pm}(\lambda_m) \sgn_{\pm}(\lambda_n + \lambda_m)
\fhat(\lambda_n)\fhat(\lambda_m) e^{i(\lambda_n + \lambda_m)\cdot}.
\end{eqnarray*}
By symmetry in $n$ and $m$, the above quantity is also equal to
\begin{equation*}
2 \sum_{n \in \mathbb{Z}} \sum_{m \in \mathbb{Z}} -\sgn_{\pm}(\lambda_n) \sgn_{\pm}(\lambda_n + \lambda_m) \fhat(\lambda_n)
\fhat(\lambda_m) e^{i(\lambda_n + \lambda_m)\cdot}.
\end{equation*}
Averaging these two expressions gives that
\begin{equation*}
2H_{\pm}(fH_{\pm}(f)) = \sum_{n \in \mathbb{Z}} \sum_{m \in \mathbb{Z}} -\big(\sgn_{\pm}(\lambda_n) + \sgn_{\pm}(\lambda_m)\big) \sgn_{\pm}(\lambda_n + \lambda_m) \fhat(\lambda_n) \fhat(\lambda_m) e^{i(\lambda_n + \lambda_m)\cdot}.
\end{equation*}
It follows that
\begin{equation*}
f^2 + 2H_{\pm}(f H_{\pm}(f)) = \sum_{n \in \mathbb{Z}} \sum_{m \in \mathbb{Z}} \bigg(1 - \big(\sgn_{\pm}(\lambda_n) + \sgn_{\pm}(\lambda_m)\big) \sgn_{\pm}(\lambda_n + \lambda_m)\bigg)\fhat(\lambda_n) \fhat(\lambda_m) e^{i (\lambda_n + \lambda_m) \cdot}.
\end{equation*}
It is trivial to see that
\begin{equation*}
1 - (\sgn_{\pm}(\lambda_n) + \sgn_{\pm}(\lambda_m))\sgn_{\pm}(\lambda_n + \lambda_m) = -\sgn_{\pm}(\lambda_n)\sgn_{\pm}(\lambda_m).
\end{equation*}
Consequently,
\begin{eqnarray*}
f^2 + 2H_{\pm}(f H_{\pm}(f)) &=& \sum_{n \in \mathbb{Z}} \sum_{m \in \mathbb{Z}} (-\sgn_{\pm}(\lambda_n) \sgn_{\pm}(\lambda_m)) \fhat(\lambda_n) \fhat(\lambda_m) e^{i (\lambda_n + \lambda_m)\cdot}\\*
&=& \Big(\sum_{n \in \mathbb{Z}} (-i \sgn_{\pm}(\lambda_n))\fhat(\lambda_n)e^{i\lambda_n \cdot}\Big)^2\\
&=& (H_{\pm}(f))^2.
\end{eqnarray*}
\end{proof}

Using this identity, the following vector-valued operator bound for the modified Hilbert transforms may be established:
\begin{thm} \label{stephill2}
For any given $\| k \in \mathbb{N}$, let $\| (f_j)_{j \in \mathbb{N}}$ be a sequence of functions in $\| S^{2^k}$ that satisfy the separation condition uniformly with separation constant $\alpha > 0$.  Then
\begin{equation*}
\Big\Vert \Big( \sum_{j \in \mathbb{N}} |H_{\pm}f_j|^2\Big)^{\frac{1}{2}} \Big\Vert_{S^{2^k}} \lesssim_{\alpha} \Big\Vert \Big( \sum_{j \in \mathbb{N}} |f_j|^2\Big)^{\frac{1}{2}} \Big\Vert_{S^{2^k}}.
\end{equation*}
\end{thm}
\begin{proof}
Fix any $\| \epsilon > 0$ and for each $\| j \in \mathbb{N}$, choose $\| N_j \in \mathbb{N}$ such that $\| f_{j, N_j}$ is a trigonometric polynomial approximating $\| f_j$ in the sense that $\| \Vert f_{j, N_j} - f_j \Vert_{S^{2^k}} < \frac{\epsilon}{2^{\frac{j}{2}}}$.  Now, note that
\begin{eqnarray*}
&& \Big\Vert \Big( \sum_{j \in \mathbb{N}} | f_{j, N_j} |^2 \Big)^{\frac{1}{2}} - \Big( \sum_{j \in \mathbb{N}} |f_j|^2 \Big)^{\frac{1}{2}} \Big\Vert_{S^{2^k}}\\*
&\leq& \Big\Vert \Big(\sum_{j \in \mathbb{N}} | f_{j, N_j} - f_j |^2 \Big)^{\frac{1}{2}} \Big\Vert_{S^{2^k}}\\
&=& \Big\Vert \sum_{j \in \mathbb{N}} |f_{j, N_j} - f_j|^2 \Big\Vert_{S^{2^{k-1}}}^{\frac{1}{2}}\\
&\leq& \Big(\sum_{j \in \mathbb{N}} \Vert f_{j, N_j} - f_j \Vert_{S^{2^k}}^2 \Big)^{\frac{1}{2}} \mbox{ by the triangle inequality.}
\end{eqnarray*}
Consequently, by density, it suffices to assume that each $\| f_j$ is a trigonometric polynomial.  Also, without loss of generality, assume that $\| f_j$ is real valued for every $\| j \in \mathbb{N}$.  It is clear that $\| H_{\pm} f_j$ is thus also real valued from Definition \ref{aphilberttransform}.
 
The proof will proceed by induction.  To begin, assume that $\| \sum_{j \in \mathbb{N}} |f_j|^2$ is a finite sum.  The general case will then follow by monotone convergence.  By Proposition \ref{bred} and Parseval's identity,
\begin{eqnarray*}
&& \Vert ( \sum_{j \in \mathbb{N}} |H_{\pm}f_j|^2)^{\frac{1}{2}} \Vert_{S^2}^2\\*
&\leq& \Big(\frac{2\pi}{\alpha} + 1\Big) \sum_{j \in \mathbb{N}} \sum_{n \in \mathbb{N}} |-i \sgn_{\pm}(\lambda_n)\widehat{f_j}(\lambda_n)|^2\\
&=& \(\frac{2\pi}{\alpha} + 1\) \sum_{j \in \mathbb{N}} \lim_{T \rightarrow \infty} \frac{1}{2T} \int_{-T}^T |f_j|^2\\
&\leq& \(\frac{2\pi}{\alpha} + 1\) \Big\Vert \Big(\sum_{j \in \mathbb{N}} |f_j|^2\Big)^{\frac{1}{2}} \Big\Vert_{S^2}^2.
\end{eqnarray*}
The result thus follows for $k = 1$.  Now, assume that the theorem holds on $\| S^{2^k}$ for some particular $k \in \mathbb{N}$ and note that it follows \emph{a fortiori} that $\| \Vert H_{\pm}f \Vert_{S^{2^k}} \lesssim \Vert f \Vert_{S^{2^k}}$.  Using the identity of Lemma \ref{niceformula},
\begin{eqnarray*}
&& \Big\Vert \Big(\sum_{j \in \mathbb{N}} |H_{\pm}f_j|^2\Big)^{\frac{1}{2}} \Big\Vert_{S^{2^{k+1}}}\\*
&=& \Big\Vert \sum_{j \in \mathbb{N}} f_j^2 + 2H_{\pm}(f_jH_{\pm}(f_j)) \Big\Vert_{S^{2^k}}^{\frac{1}{2}}\\
&\leq& \bigg(\Big\Vert \sum_{j \in \mathbb{N}} f_j^2 \Big\Vert_{S^{2^k}} + 2 \Big\Vert \sum_{j \in \mathbb{N}} H_{\pm}(f_jH_{\pm}(f_j)) \Big\Vert_{S^{2^k}}\bigg)^{\frac{1}{2}}\\
&=& \bigg(\Big\Vert \Big(\sum_{j \in \mathbb{N}} f_j^2 \Big)^{\frac{1}{2}} \Big\Vert_{S^{2^{k+1}}}^2 + 2 \Big\Vert \sum_{j \in \mathbb{N}}  H_{\pm}(f_jH_{\pm}(f_j)) \Big\Vert_{S^{2^k}} \bigg)^{\frac{1}{2}}.
\end{eqnarray*}
By the inductive hypothesis and H\"older's inequality,
\begin{eqnarray*}
&& \Big\Vert \sum_{j \in \mathbb{N}} H_{\pm}(f_j H_{\pm}(f_j)) \Big\Vert_{S^{2^k}}\\*
&=& \Big\Vert H_{\pm}(\sum_{j \in \mathbb{N}} f_j H_{\pm}(f_j)) \Big\Vert_{S^{2^k}}\\
&\leq& \Vert H_{\pm} \Vert_{S^{2^k} \rightarrow S^{2^k}} \Big\Vert \sum_{j \in \mathbb{N}} | f_j H_{\pm}(f_j) | \Big\Vert_{S^{2^k}}\\
&\leq& \Vert H_{\pm} \Vert_{S^{2^k} \rightarrow S^{2^k}} \Big\Vert \Big(\sum_{j \in \mathbb{N}} |f_j|^2\Big)^{\frac{1}{2}} \Big( \sum_{j \in \mathbb{N}} |H_{\pm}f_j|^2\Big)^{\frac{1}{2}} \Big\Vert_{S^{2^k}}\\
&\leq& \Vert H_{\pm} \Vert_{S^{2^k} \rightarrow S^{2^k}} \Big\Vert\Big(\sum_{j \in \mathbb{N}} |f_j|^2\Big)^{\frac{1}{2}} \Big\Vert_{S^{2^{k+1}}} \Big\Vert \Big(\sum_{j \in \mathbb{N}} |H_{\pm}f_j|^2\Big)^{\frac{1}{2}} \Big\Vert_{S^{2^{k+1}}}.
\end{eqnarray*}

It thus follows that
\begin{eqnarray*}
\Big\Vert \Big(\sum_{j \in \mathbb{N}} |H_{\pm}f_j|^2\Big)^{\frac{1}{2}} \Big\Vert_{S^{2^{k+1}}} &\leq& \bigg(\Big\Vert \Big(\sum_{j \in \mathbb{N}} |f_j|^2\Big)^{\frac{1}{2}} \Big\Vert_{S^{2^{k+1}}}^2\\*
&& \quad {} + 2 \Vert H_{\pm} \Vert_{S^{2^k} \rightarrow S^{2^k}} \Big\Vert \Big(\sum_{j \in \mathbb{N}} |f_j|^2\Big)^{\frac{1}{2}} \Big\Vert_{S^{2^{k+1}}} \Big\Vert \Big(\sum_{j \in \mathbb{N}} |H_{\pm}f_j|^2\Big)^{\frac{1}{2}} \Big\Vert_{S^{2^{k+1}}} \bigg)^{\frac{1}{2}}.
\end{eqnarray*}

Using this,
\begin{equation*}
\(\frac{\|\Big\Vert\Big(\sum_{j \in \mathbb{N}} |H_{\pm}f_j|^2\Big)^{\frac{1}{2}} \Big\Vert_{S^{2^{k+1}}}}{\|\Big\Vert \Big(\sum_{j \in \mathbb{N}} |f_j|^2 \Big)^{\frac{1}{2}} \Big\Vert_{S^{2^{k+1}}}} \)^2 \leq 1 + 2 \Vert H_{\pm} \Vert_{S^{2^k} \rightarrow S^{2^k}} \frac{ \| \Big\Vert \Big(\sum_{j \in \mathbb{N}} |H_{\pm}f_j|^2 \Big)^{\frac{1}{2}} \Big\Vert_{S^{2^{k+1}}}}{\|\Big\Vert\Big(\sum_{j \in \mathbb{N}} |f_j|^2\Big)^{\frac{1}{2}} \Big\Vert_{S^{2^{k+1}}}}. 
\end{equation*}
Consequently, 
\begin{equation*}
\Big\Vert \Big(\sum_{j \in \mathbb{N}} |H_{\pm}f_j|^2 \Big)^{\frac{1}{2}} \Big\Vert_{S^{2^{k+1}}} \leq \( \Vert H_{\pm} \Vert_{S^{2^k} \rightarrow S^{2^k}} + \sqrt{\Vert H_{\pm} \Vert_{S^{2^k} \rightarrow S^{2^k}}^2 + 1} \) \Big\Vert \Big(\sum_{j \in \mathbb{N}} |f_j|^2\Big)^{\frac{1}{2}} \Big\Vert_{S^{2^{k+1}}}.
\end{equation*}
\end{proof}

\section{Littlewood--Paley Theory for Almost Periodic Functions} \label{seclp}
Boundedness of the maximal dyadic summation operator will involve bounding a square function and to this end, in this section a theorem of Littlewood--Paley type for almost periodic functions will be proved.  To begin with, a standard Littlewood--Paley theorem for $\| \mathbb{R}$ is stated; this result is a special case of Theorem 5.1.2 from \cite{grafakosclassical}, p. 343.
\begin{thm} [Littlewood--Paley on $\| L^p(\mathbb{R})$] \label{LPR}
Let $\| \psi \in \mathcal{S}(\mathbb{R})$ be such that $\| \widehat{\psi}(0) = 0$ and for each $k \in \mathbb{N}$, define $\| \psi_k \coloneq 2^k \psi(2^k \cdot)$ so that $\| \widehat{\psi_k} = \widehat{\psi}(2^{-k} \cdot)$.  Then for all $\| f \in L^p(\mathbb{R})$, $\| p \in (1, \infty)$,
\begin{equation*}
\Big\Vert \Big(\sum_{k \in \mathbb{N}} |f* \psi_k|^2\Big)^{\frac{1}{2}} \Big\Vert_{L^p(\mathbb{R})} \lesssim \Vert f \Vert_{L^p(\mathbb{R})}.
\end{equation*}
\end{thm}
For a proof of this result, see \cite{grafakosclassical}, pp. 344--345.

The following analogous result for $\| S^p$, $\| p \in (1, \infty)$, will be proved:
\begin{thm} [Littlewood--Paley on $\| S^p$] \label{StepLP}
Let $\| \psi \in \mathcal{S}(\mathbb{R})$ be such that $\| \widehat{\psi}(0) = 0$ and for each $k \in \mathbb{N}$, define $\| \psi_k \coloneq 2^k \psi(2^k \cdot)$ so that $\| \widehat{\psi_k} = \widehat{\psi}(2^{-k} \cdot)$.  Then for all $\| f \in S^p$, $\| p \in (1, \infty)$,
\begin{equation*}
\Big\Vert \Big(\sum_{k \in \mathbb{N}} |f* \psi_k|^2\Big)^{\frac{1}{2}} \Big\Vert_{S^p} \lesssim \Vert f \Vert_{S^p}.
\end{equation*}
\end{thm}
\begin{proof}
For each $\| k \in \mathbb{N}$, $\| l \in \mathbb{N}$ and $\| j \in \mathbb{Z}$, write $\| \psi_k^{(l)} \coloneq \psi_k \chi_{(-2^l, -2^{l-1}] \cup [2^{l-1}, 2^l)}$, $\psi_k^{(0)} \coloneq \psi_k \chi_{(-1,1)}$ and $f_j \coloneq f \chi_{[j, j+1)}$.  Then by the triangle inequality,
\begin{eqnarray*}
&& \Big\Vert \Big(\sum_{k \in \mathbb{N}} | f* \psi_k |^2\Big)^{\frac{1}{2}} \Big\Vert_{S^p}\\*
&=& \Big\Vert\Big(\sum_{k \in \mathbb{N}} \Big| \Big(\sum_{l = 0}^{\infty} \Big(\sum_{j \in \mathbb{Z}}f_j\Big) * \psi_k^{(l)}\Big) \Big|^2 \Big)^{\frac{1}{2}} \Big\Vert_{S^p}\\
&\leq& \sum_{l = 0}^{\infty} \Big\Vert \sum_{j \in \mathbb{Z}} \Big(\sum_{k \in \mathbb{N}} \Big|\psi_k^{(l)} * f_j \Big|^2 \Big)^{\frac{1}{2}} \Big\Vert_{S^p}.
\end{eqnarray*}

For fixed $\| x \in \mathbb{R}$ and $\| l \in \mathbb{N} \cup \{ 0 \}$, define
\begin{equation*}
I(x, l) \coloneq \int_x^{x+1} \Big(\sum_{j \in \mathbb{Z}} \Big(\sum_{k \in \mathbb{N}} | \psi_k^{(l)} * f_j |^2 \Big)^{\frac{1}{2}}\Big)^p
\end{equation*}
and note that $\| \Big\Vert \Big(\sum_{k \in \mathbb{N}} | f*\psi_k|^2\Big)^{\frac{1}{2}} \Big\Vert_{S^p} \leq \sum_{l=0}^{\infty} \sup_{x \in \mathbb{R}} (I(x, l))^{\frac{1}{p}}$.  Now, for $\| j \in \mathbb{Z}$ and $l \in \mathbb{N}$, $\| \supp (\psi_k^{(l)} * f_j) \subseteq [-2^l + j, -2^{l-1} + (j + 1)] \cup [2^{l-1} + j, 2^l + (j+1)]$ and $\| \supp (\psi_k^{(0)} * f_j) \subseteq [j - 1, j + 2]$.  As such, for each $\| l \in \mathbb{N}$ and $x \in \mathbb{R}$, define
\begin{eqnarray*}
K_l(x) &\coloneq& \big\{ j \in \mathbb{Z} : \big|\big([-2^l + j, -2^{l-1} + (j + 1)] \cup [2^{l-1} + j, 2^l + (j+1)]\big) \cap [x, x + 1]\big| \neq 0 \big\},\\
K_0(x) &\coloneq& \{ j \in \mathbb{Z} : |[j - 1, j + 2] \cap [x, x + 1]| \neq 0 \},
\end{eqnarray*}
and note that $\| |K_l(x)| \leq 2^l + 4$ for any $\| l \in \mathbb{N} \cup \{0\}$ and $x \in \mathbb{R}$.  Applying H\"older's inequality to the above sum in $j$ and using this fact, it follows that for $\| x \in \mathbb{R}$ and $\| l \in \mathbb{N} \cup \{ 0 \}$,
\begin{eqnarray*}
I(x, l) &\leq& \int_x^{x+1} (2^l + 4)^{\frac{p}{p'}} \sum_{j \in K_l} \Big(\sum_{k \in \mathbb{N}} |\psi_k^{(l)} * f_j|^2\Big)^{\frac{p}{2}}\\*
&\leq& (2^l + 4)^{\frac{p}{p'} + 1} \sup_{j \in K_l} \int_x^{x+1} \Big(\sum_{k \in \mathbb{N}} |\psi_k^{(l)} * f_j |^2 \Big)^{\frac{p}{2}}.
\end{eqnarray*}

This term will be bounded separately for $l = 0$ and $l \in \mathbb{N}$ with the former case making the dominant contribution owing to the Schwartz decay of the function $\psi$.  Indeed, first of all, fixing $l \in \mathbb{N}$, recall that $\| \supp(\psi_k^{(l)}) \subseteq [-2^l, -2^{l-1}] \cup [2^{l-1}, 2^l]$ for each $k \in \mathbb{N}$.  Using the fact that $\| \psi$ is Schwartz, it follows that for each $\| N \in \mathbb{N}$, $k \in \mathbb{N}$ and $x \in \mathbb{R}$,
\begin{equation*}
|\psi_k^{(l)}(x)| \lesssim_N \frac{1}{2^{(N-1)k + N(l-1)}}.
\end{equation*}
Consequently, for $\| x \in \mathbb{R}$ and $\| l \in \mathbb{N}$,
\begin{eqnarray*}
I(x, l) &\lesssim_N& (2^l + 4)^{\frac{p}{p'}+1} \sup_{j \in K_l} \int_x^{x+1} \( \sum_{k \in \mathbb{N}} \Big|
\frac{1}{2^{(N-1)k + N(l-1)}} \int_{\mathbb{R}} f_j \Big|^2\)^{\frac{p}{2}}\\
&\leq& (2^l + 4)^{\frac{p}{p'} + 1} \(\frac{1}{2^{N(l-1)}}\)^p \sup_{j \in K_l} \Vert f \Vert_{L^1([j, j+1])}^p
\(\sum_{k \in \mathbb{N}} \frac{1}{4^{(N-1)k}}\)^{\frac{p}{2}}.
\end{eqnarray*}

Choosing $\| N \geq 2$ to make the sum in $k$ convergent, it follows that
\begin{equation*}
I(x, l) \lesssim_N (2^l + 4)^{\frac{p}{p'} + 1} \(\frac{1}{2^{N(l-1)}}\)^p \Vert f \Vert_{S^1}^p \leq (2^l + 4)^{\frac{p}{p'} + 1} \(\frac{1}{2^{N(l-1)}}\)^p \Vert f \Vert_{S^p}^p.
\end{equation*}
In particular, choosing $N$ to be sufficiently large, it can be seen that
\begin{equation*}
\sum_{l=1}^{\infty} \sup_{x \in \mathbb{R}} (I(x, l))^{\frac{1}{p}} \lesssim \Vert f \Vert_{S^p}.
\end{equation*}

For the case of $\| l = 0$, consider that by the triangle inequality,
\begin{equation*}
(I(x, 0))^{\frac{1}{p}} \lesssim \(\sup_{j \in K_l} \int_{\mathbb{R}} \Big(\sum_{k \in \mathbb{N}} |\psi_k * f_j|^2\Big)^{\frac{p}{2}}\)^{\frac{1}{p}} + \(\sup_{j \in K_l} \int_x^{x+1} \Big(\sum_{k \in \mathbb{N}} |(\psi_k - \psi_k^{(0)}) * f_j|^2\Big)^{\frac{p}{2}}\)^{\frac{1}{p}}.
\end{equation*}
Using Theorem \ref{LPR}, the first term here can be bounded by a constant multiple of $\| \sup_{j \in K_l} \Vert f \Vert_{L^p([j, j+1])}$ which is less than or equal to $\| \Vert f \Vert_{S^p}$.  The second term can be bounded by a constant multiple of
\begin{equation*}
\sum_{l = 1}^{\infty} \(\sup_{j \in K_l} \int_x^{x+1} \Big(\sum_{k \in \mathbb{N}} |\psi_k^{(l)} * f_j |^2\Big)^{\frac{p}{2}}\)^{\frac{1}{p}},
\end{equation*}
which, as before, is bounded by a constant multiple of $\| \Vert f \Vert_{S^p}$.  It is thus the case that
\begin{equation*}
(I(x, 0))^{\frac{1}{p}} \lesssim \Vert f \Vert_{S^p}
\end{equation*}
and so the theorem follows.
\end{proof}

\section{\texorpdfstring{Proof of Theorem \ref{mainthm}}{Proof of Theorem 1.2}} \label{secmainproof}
By a straightforward density argument, to prove Theorem \ref{mainthm}, it suffices to show that
\begin{equation*}
\Vert S^{*}f \Vert_{S^{2^k}} \lesssim_{\alpha_f} \Vert f \Vert_{S^{2^k}}
\end{equation*}
for all $f \in \mathcal{P}$.

To begin with, choose any $\phi \in \mathcal{S}(\mathbb{R})$ such that $\supp(\widehat{\phi}) \subseteq [-1, 1]$, $\widehat{\phi}(\xi) \in [0, 1]$ for all $\xi \in \mathbb{R}$ and $\widehat{\phi}(\xi) = 1$ for $\xi \in [-\frac{1}{2}, \frac{1}{2}]$.  For each $k \in \mathbb{N}$, define $\phi_k = 2^k \phi(2^k \cdot)$, so that $\widehat{\phi_k} = \widehat{\phi}(2^{-k} \cdot)$.  Then define a smoothed version of the operator $S_k$ as
\begin{equation*}
R_k f \coloneq \sum_{n \in \mathbb{Z}} \fhat(\lambda_n) \widehat{\phi_k}(\lambda_n) e^{i\lambda_n \cdot}.
\end{equation*}

The maximal operator $S^{*}$ may now be trivially pointwise bounded as follows:
\begin{equation*}
S^{*}f \leq \sup_{j \in \mathbb{N}} |S_jf - R_jf| + \sup_{j \in \mathbb{N}} |R_j f| \leq \Big(\sum_{j \in \mathbb{N}} |S_j f - R_j f|^2 \Big)^{\frac{1}{2}} + \sup_{j \in \mathbb{N}} |R_j f|.
\end{equation*}
The proof of Theorem \ref{mainthm} has thus been reduced to establishing boundedness of a square function and smoothed maximal operator.  The smoothed maximal operator may be bounded using a similar technique to that used to prove Theorem \ref{StepLP}.  Indeed, for each $\| k \in \mathbb{N}$, $\| l \in \mathbb{N}$ and $\| j \in \mathbb{Z}$, write $\| \phi_k^{(l)} \coloneq \phi_k \chi_{(-2^l, -2^{l-1}] \cup [2^{l-1}, 2^l)}$, $\| \phi_k^{(0)} \coloneq \phi_k \chi_{(-1,1)}$ and $\| f_j \coloneq f \chi_{[j, j+1)}$.  Then by the triangle inequality,
\begin{equation*}
\Vert \sup_{k \in \mathbb{N}} |R_k(f)| \Vert_{S^p} \leq \sum_{l=0}^{\infty} \Big\Vert \sum_{j \in \mathbb{Z}} \sup_{k \in \mathbb{N}} | \phi_k^{(l)} * f_j | \Big\Vert_{S^p}.
\end{equation*}

For fixed $x \in \mathbb{R}$ and $l \in \mathbb{N} \cup \{0\}$, define
\begin{equation*}
I(x, l) \coloneq \int_x^{x+1} \Big( \sum_{j \in \mathbb{Z}} \sup_{k \in \mathbb{N}} | \phi_k^{(l)} * f_j | \Big)^p
\end{equation*}
and note that $\| \Vert \sup_{k \in \mathbb{N}} |R_k(f)| \Vert_{S^p} \leq \sum_{l=0}^{\infty} \sup_{x \in \mathbb{R}} (I(x, l))^{\frac{1}{p}}$.  Now, for $j \in \mathbb{Z}$ and $l \in \mathbb{N}$, $\| \supp (\phi_k^{(l)} * f_j) \subseteq [-2^l + j, -2^{l-1} + (j + 1)] \cup [2^{l-1} + j, 2^l + (j+1)]$ and $\| \supp (\phi_k^{(0)} * f_j) \subseteq [-1 + j, 1 + (j + 1)] = [j - 1, j + 2]$.  As such, for each $\| l \in \mathbb{N}$ and $x \in \mathbb{R}$, define
\begin{eqnarray*}
K_l(x) &\coloneq& \big\{ j \in \mathbb{Z} : \big|\big([-2^l + j, -2^{l-1} + (j + 1)] \cup [2^{l-1} + j, 2^l + (j+1)]\big) \cap [x, x + 1] \big| \neq 0 \big\},\\
K_0(x) &\coloneq& \{ j \in \mathbb{Z} : \left|[j - 1, j + 2] \cap [x, x + 1] \right| \neq 0 \},
\end{eqnarray*}
and note that $\| |K_l(x)| \leqslant 2^l + 4$ for any $\| l \in \mathbb{N} \cup \{0\}$ and $x \in \mathbb{R}$.  Now, applying H\"older's inequality to the above sum in $j$ and using this fact, it follows that for $x \in \mathbb{R}$, $l \in \mathbb{N} \cup \{0\}$,
\begin{eqnarray*}
I(x, l) &\leq& \int^{x+1}_x (2^l + 4)^{\frac{p}{p'}} \sum_{j \in K_l} \sup_{k \in \mathbb{N}} | \phi_k^{(l)} * f_j (s) |^p  \, ds \\*
&\leq& (2^l +4)^{\frac{p}{p'} + 1} \sup_{j \in K_l} \int_x^{x+1} \sup_{k \in \mathbb{N}} | \phi_k^{(l)} * f_j (s) |^p \, ds.
\end{eqnarray*}
As in the proof of Theorem \ref{StepLP}, this term will be bounded separately for $l = 0$ and $l \in \mathbb{N}$.  For $l \in \mathbb{N}$, recall that $\| \supp{\phi_k^{(l)}} \subseteq [-2^l, -2^{l-1}]\cup[2^{l-1}, 2^l]$ for each $k \in \mathbb{N}$, and so by the fact that $\phi$ is Schwartz, it may be deduced that for each $N \in \mathbb{N}$, $k \in \mathbb{N}$ and $x \in \mathbb{R}$,
\begin{equation*}
|\phi_k^{(l)}(x)| \lesssim_N \frac{1}{2^{(N-1)k + N(l-1)}}
\end{equation*}
which is bounded by a constant multiple of
\begin{equation*}
\frac{1}{2^{N(l-1)}}
\end{equation*}
as $\| k \in \mathbb{N}$.  Consequently, for $x \in \mathbb{R}$ and $l \in \mathbb{N}$,
\begin{eqnarray*}
I(x, l) &\lesssim_N& (2^l + 4)^{\frac{p}{p'} + 1} \sup_{j \in K_l} \int_x^{x+1} \(\frac{1}{2^{N(l-1)}} \int_{\mathbb{R}} |f_j(t)| \, dt \)^p \, ds \\*
&\lesssim& \frac{(2^l+4)^{\frac{p}{p'} + 1}}{2^{Np(l-1)}} \Vert f \Vert_{S^p}^p.
\end{eqnarray*}
It follows by selecting sufficiently large $N$ that
\begin{equation*}
\sum_{l=1}^{\infty} \sup_{x \in \mathbb{R}} (I(x, l))^{\frac{1}{p}} \lesssim \Vert f \Vert_{S^p}.
\end{equation*}
For the case of $\| l = 0$, consider that
\begin{eqnarray*}
(I(x, 0))^{\frac{1}{p}} &\lesssim& \(\sup_{j \in K_0} \int_x^{x+1} \sup_{k \in \mathbb{N}} | |\phi_k| * |f_j|(s) |^p \, ds \)^{\frac{1}{p}}\\*
&\leq& \sup_{j \in K_0} \Vert M(f_j)(s) \Vert_{L^p(\mathbb{R})}
\end{eqnarray*}
where $M$ is the Hardy--Littlewood maximal function on $\mathbb{R}$, owing to the fact that $\phi$ is Schwartz.  By boundedness of $M$ on $L^p(\mathbb{R})$, it follows that
\begin{equation*}
\sup_{x \in \mathbb{R}} (I(x, 0))^{\frac{1}{p}} \lesssim \sup_{j \in K_0} \Vert f \Vert_{L^p([j, j+1])} \leq \Vert f \Vert_{S^p}
\end{equation*}
and it can thus be concluded that the smoothed maximal operator is bounded.

To prove Theorem \ref{mainthm}, it remains to establish boundedness of the square function.  First note that for each $j \in \mathbb{N}$ and $\lambda \in \mathbb{R}$,
\begin{equation*}
\widehat{(S_jf)}(\lambda) = \chi_{[-2^j, 2^j]}(\lambda) \fhat(\lambda) = \frac{i}{2}(-i\sgn_{+}(\lambda + 2^j)\widehat{(e^{2\pi i 2^j .}f)}(\lambda + 2^j) + i \sgn_{-}(\lambda - 2^j)\widehat{(e^{-2\pi i 2^j .}f)}(\lambda - 2^j)).
\end{equation*}
Consequently,
\begin{equation*}
S_jf(x) = \frac{i}{2}(e^{-2 \pi i 2^jx}H_{+}(e^{2\pi i 2^j \cdot} f)(x) - e^{2\pi i 2^jx}H_{-}(e^{-2\pi i 2^j \cdot}f)(x))
\end{equation*}
for all $\| x \in \mathbb{R}$.

Define $\| \psi \in \mathcal{S}(\mathbb{R})$ so that $\| \widehat{\psi} = \widehat{\phi}(\tfrac{1}{2} \cdot) - \widehat{\phi}$ and note that $\| \chi_{[-1, 1]} - \widehat{\phi} = \chi_{[-1, 1]} \widehat{\psi}$.  Defining for each $k \in \mathbb{N}$, $\psi_k \coloneq 2^k \psi(2^k \cdot)$ so that $\widehat{\psi_k} = \widehat{\psi}(2^{-k} \cdot)$, it follows by Theorem \ref{bredext}, the above observation and Theorems \ref{stephill2} and \ref{StepLP} that
\begin{eqnarray*}
&& \Vert ( \sum_{j \in \mathbb{N}} | S_j f - R_j f |^2 )^{\frac{1}{2}} \Vert_{S^p}\\*
&=& \Vert (\sum_{j \in \mathbb{N}} | S_j (f * \psi_j) |^2 )^{\frac{1}{2}} \Vert_{S^p}\\
&=& \Vert (\sum_{j \in \mathbb{N}} | e^{-2 \pi i 2^j \cdot} H_{+}(e^{2 \pi i 2^j \cdot} (f * \psi_j)) - e^{2 \pi i 2^j \cdot} H_{-}(e^{-2 \pi i 2^j \cdot} (f * \psi_j)) |^2)^{\frac{1}{2}} \Vert_{S^p}\\
&\lesssim& \Vert ( \sum_{j \in \mathbb{N}} |f * \psi_j|^2)^{\frac{1}{2}} \Vert_{S^p}\\
&\lesssim& \Vert f \Vert_{S^p}
\end{eqnarray*}
and so the proof of Theorem \ref{mainthm} is complete.

\section{Further Remarks} \label{secfurtherremarks}
It is noted that it only remains to generalise the boundedness of the Hilbert transform established in Theorem \ref{stephill2} to $S^p$ for all $p \in (1, \infty)$ to prove Theorem \ref{mainthm} on the same spaces.  The analogous theorem in $L^p$ generalises from the case of $p = 2^k$, $k \in \mathbb{N}$, automatically by interpolation, but a suitable approach to interpolation in the Stepanov setting is not apparent.  Further, there is a certain degree of subtlety concerning the problem of boundedness of the Hilbert transform in the setting of Stepanov norms.  In particular, the Hilbert transform fails to be bounded on the space of all $L^p_{\loc}(\mathbb{R})$ functions with finite $S^p$ norm (which, as was remarked before, is a strictly larger space than $S^p$).  To see this, note that this space can be shown to be equivalent to the ``amalgam'' space $\| (L^p, \ell^{\infty})$, where 
\begin{equation*}
\Vert \cdot \Vert_{(L^p, \ell^{\infty})} \coloneq \( \sup_{x \in \mathbb{Z}} \int_x^{x+1} |\cdot|^p \, \)^{\frac{1}{p}}.
\end{equation*}
As was proved in \cite[Thm. 2.6]{amalgams}, the pre-dual space of $\| (L^p, \ell^{\infty})$ is $\| (L^{p'}, \ell^1)$, which is the space defined with the natural norm, 
\begin{equation*}
\Vert \cdot \Vert_{(L^{p'}, \ell^1)} \coloneq \sum_{n \in \mathbb{Z}} \( \int_n^{n+1} | \cdot |^{p'} \, \)^{\frac{1}{p'}}.
\end{equation*}
From this fact, if the Hilbert transform is bounded on $\| (L^p, \ell^{\infty})$, it is also bounded on $\| (L^{p'}, \ell^1)$.  This can be shown to be false by calculating that $\| H(\chi_{[0,1]})(x) = \frac{1}{\pi} \log \frac{|x|}{|x-1|}$ for all $x \in \mathbb{R}$ and considering the relevant norms.\footnote{The author would like to thank Michael Cowling for suggesting this approach to the Hilbert transform on the Stepanov spaces.}

It should be remarked that boundedness of the Hilbert transform on $S^p$ is claimed in \cite{koizumi2}.  However, the proof there unfortunately seems to contain an error.\footnote{In particular, when bounding the ``$\| J_3$'' term, the triangle inequality is used erroneously on moving from the second line to the third.}  Further, the claimed theorem is the demonstrably false boundedness of the Hilbert transform on the amalgam spaces mentioned above and thus it seems unlikely that the proposed scheme of proof can be repaired.

\renewcommand{\bibname}{Bibliography} 

\bibliographystyle{adb}
\bibliography{apdyadic}

\end{document}